\documentclass{amsart}
\usepackage{graphicx}
\usepackage{ulem}
\usepackage{slashed}
\usepackage{tikz}
\usepackage{amssymb}

\newtheorem{theorem}{Theorem}[section]

\newtheorem{theoremL}{Theorem}[section]
\newtheorem{lemma}[theoremL]{Lemma}

\newtheorem{theoremC}{Theorem}[section]
\newtheorem{conjecture}[theoremC]{Conjecture}

\newtheorem{theoremD}{Theorem}[section]
\newtheorem{definition}[theoremD]{Definition}

\newtheorem{theoremCr}{Theorem}[section]
\newtheorem{corollary}[theoremCr]{Corollary}

\numberwithin{equation}{section}

\newcommand*\circled[1]{\tikz[baseline=(char.base)]{
    \node[shape=circle,draw,inner sep=1pt] (char) {#1};}}

\begin{document}

\title{On twin prime distribution and associated biases}

\author{Shaon Sahoo}
\address{Indian Institute of Technology, Tirupati 517506, India}
\email{shaon@iittp.ac.in}

\keywords{Twin primes, Arithmetic progression, Chebyshev's bias}

\begin{abstract}{
A modified totient function ($\phi_2$) is seen to play a significant role in the study of the twin prime distribution.
The function is defined as $\phi_2(n):=\#\{a\le n ~\vert ~\textrm{$a(a+2)$ is coprime to $n$}\}$
and is shown here to have following product form: $\phi_2(n) = n (1-\frac{\theta_n}{2}) \prod_{p>2,~p\vert n}(1-\frac 2 p)$, where $p$ denotes a prime and
$\theta_n = 0$ or $1$ for odd or
even $n$ respectively. Using this function it is proved for a given $n$ that there always exists a number $m > n$ so 
that $(p, m(m + 2)) = 1$ for every prime $p \le n$. We also establish a Legendre-type formula for the twin prime
counting function in the following form: $\pi_2(x) - \pi_2(\sqrt{x})  = \sum_{ab\vert P(\sqrt{x})}\mu(ab) \left[\frac{x-l_{a,b}}{ab}\right]$, where
$P(z)=\prod_{p\le z}p$ and $a$ is always odd. Here $l_{a,b}$ is the lowest positive integer so that $a\vert l_{a,b}$ and $b\vert (l_{a,b}+2)$.

In the latter part of this work, we discussion three different types of biases in the distribution of twin primes. The first two biases are similar to
the biases in primes as reported by Chebyshev, and Oliver and Soundararajan. Our third reported bias is on the difference ($D$) between (the first members of) 
two consecutive twin primes; it is observed that $D\pm1$ is more likely to be a prime than an odd composite number.
}
\end{abstract}

\maketitle
\tableofcontents

\section{Introduction}
\label{sec1}

Important proven results are scarce in the study of twin primes. The statements about the infinitude of twin primes and corresponding distribution are the well known 
conjectures in this field. Although unproven so far, there are some important breakthroughs towards proving these conjectures. Chen's work \cite{chen73} and 
more recent work by Zhang \cite{zhang14} may be noted in this regard. 

In this work we study some new aspects of the distribution of twin primes; we especially focus on studying the twin primes in arithmetic progressions. 
Many important results are known when primes are studied in residue classes. For example, Dirichlet's work (1837) shows that, if $(a,q) = 1$, there exists 
infinite number of primes $p$ so that $p\equiv a~(\text{mod}~q)$. In analogy with these studies, one can also study twin primes after classifying them according 
to the residues of their first members. In this study of twin primes, a modified totient function ($\phi_2$) is seen to play a role similar to that of the Euler's 
totient function in the study of primes. We analyze this modified totient function and an associated modified M\"{o}bius function ($\mu_2$). 

For the twin primes, we also discuss in this work a sieve similar to the sieve of Eratosthenes and a formula similar to the Legendre's formula for the prime counting
function. We also provide a new heuristics for the twin prime conjecture using the modified totient function $\phi_2$.

While studying primes in residue classes, one of the most interesting results that came out is the biases or, as popularly known, the prime number 
races \cite{granville06}. Chebyshev (1853) first noted that certain residue classes are surprisingly more preferred over others. For example, if $\pi(x;q;a)$ 
denotes the number of primes $\le x$ in the residue class $a$ (mod $q$), it is seen that $\pi(x;4;3)>\pi(x;4;1)$ in most occasions.
A detailed explanation of this bias (Chebyshev's) is given by Rubinstein and Sarnak in their 1994 paper \cite{rubinstein94}.
More recently (2016), a different type of bias was reported by Oliver and Soundararajan \cite{oliver16}. They found that 
a prime of a certain residue class is more likely to be followed by a prime of another residue class. A 
conjectural explanation was presented to understand such bias. 

In the last part of this paper we discuss similar biases in the distribution of the twin prime pairs. In addition, we also report a third type of bias:
if $(\hat{p}_n-\hat{p}_{n-1})$ represents the difference between the first members of two consecutive twin prime pairs, then $(\hat{p}_n-\hat{p}_{n-1}\pm1)$ 
are more likely to be prime numbers than to be composites. We also find that the possibility of $(\hat{p}_n-\hat{p}_{n-1}-1)$ to be a prime is more 
than that of $(\hat{p}_n-\hat{p}_{n-1}+1)$ to be a prime.   
The numerical results presented here are based on the analysis of the first 500 million primes ($\pi(x_0)=5\times 10^8$) and corresponding about 30 
million twin prime pairs ($\pi_2(x_0)=29,981,546$). Here $x_0$ is taken as the last prime we consider, i.e. 11,037,271,757.

Besides these studies, we also state a conjecture regarding
the arithmetic progression of twin primes in the line with the Green-Tao theorem and give some numerical examples of this type of progression.

\section{Notations and definitions}
\label{sec2}

In this paper the letters $i,j,a,b,q,n$ and $N$ represent positive integers. The letters $x,y,z$ and $c$ represent real numbers, and $p$ always represents 
a prime number 
with $p_n$ being the $n$-th prime number. Here $\hat{p}$ represents a twin prime pair whose first member is $p$; $\hat{p}_n$ denotes the $n$-th twin prime pair.
Unless mentioned otherwise, any statement involving a twin prime pair will be interpreted with respected to the first member of the pair, e.g., 
$\hat{p} \le x$ means that the first prime of the pair is less than or equal to $x$, or $\hat{p} \equiv a$ (mod $q$) means that the first member of the 
pair is congruent to $a$ (mod $q$).

We denote gcd of $a$ and $b$ by $(a,b)$. In this paper we denote a reduced residue class of $q$ by a number $a$ if ($a,q$) = 1 and $0<a<q$. A pair of integers are 
called a {\it coprime pair} to a third integer if both the integers are individually coprimes to the third integer; e.g., 6 and 14 form a coprime pair to 25. 
A coprime pair to an integer is called a {\it twin coprime pair} if the integers in the pair are separated by 2. 

The functions $\mu(n)$ and $\phi(n)$ respectively denote the M\"{o}bius function and Euler's totient function. The functions $\omega(n)$ and 
$\omega_o(n)$ denote respectively the number of prime divisors and odd prime divisors of $n$ (for odd $n$, $\omega(n) = \omega_o(n)$, and for even $n$, 
$\omega(n) = \omega_o(n)+1$). \\

\noindent {\it Definitions:}\\
\indent $\phi_2(q):=\#\{a\le q~|(a,q)=1 ~{\rm and}~ (a+2,q)=1\}$.\\
\indent $\mu_2(n):=\mu(n)2^{\omega_o(n)}$.\\ 
\indent $\pi(x):= \#\{p \le x \}$.\\ 
\indent $\pi(x;q;a):= \#\{p \le x ~|p\equiv a (\text{mod}~q)\}$.\\
\indent $\pi(x;q;a_i|a_j):= \#\{p_n \le x ~| p_n\equiv a_i (\text{mod}~q)~\text{and}~p_{n+1}\equiv a_j (\text{mod}~q)\}$.\\
\indent $\delta(x;q;a):= \frac{\pi(x;q;a)}{\pi(x)}$.\\
\indent $\delta(x;q;a_i|a_j):= \frac{\pi(x;q;a_i|a_j)}{\pi(x;q;a_i)}$.\\
\indent $\Delta(x;q;a_i,a_j):= \pi(x;q;a_i) - \pi(x;q;a_j)$.\\
\indent $\bar{\Delta}(x;q;a_i,a_j):= \Delta(x;q;a_i,a_j)\frac{ln(x)}{\sqrt{x}}$.\\
\indent $\pi^{\pm}(x):= \#\{p_n \le x~| (p_{n+1}-p_n \pm 1) ~\text{is a prime} \}$\\
\indent $\delta^{\pm}(x):=\frac{\pi^{\pm}(x)}{\pi(x)}$\\ 
\indent $\pi_2(x):=\#\{\hat{p} \le x\}$.\\
\indent $\pi_2(x;q;a):=\#\{\hat{p} \le x ~| \hat{p}\equiv a (\text{mod}~q)\}$.\\
\indent $\pi_2(x;q;a_i|a_j):= \#\{\hat{p}_n \le x ~| \hat{p}_n\equiv a_i (\text{mod}~q)~\text{and}~\hat{p}_{n+1}\equiv a_j (\text{mod}~q)\}$.\\
\indent $\delta_2(x;q;a):= \frac{\pi_2(x;q;a)}{\pi_2(x)}$.\\
\indent $\delta_2(x;q;a_i|a_j):= \frac{\pi_2(x;q;a_i|a_j)}{\pi_2(x;q;a_i)}$.\\
\indent $\Delta_2(x;q;a_i,a_j):= \pi_2(x;q;a_i) - \pi_2(x;q;a_j)$. \\
\indent $\bar{\Delta}_2(x;q;a_i,a_j):= \Delta_2(x;q;a_i,a_j)\frac{ln^2(x)}{10\sqrt{x}}$.\\
\indent $\pi^{\pm}_2(x):= \#\{\hat{p}_n \le x~| (\hat{p}_{n+1}-\hat{p}_n \pm 1) ~\text{is a prime} \}$\\
\indent $\delta^{\pm}_2(x):=\frac{\pi^{\pm}_2(x)}{\pi_2(x)}$\\

\section{Arithmetic progressions and twin primes}
\label{sec3}

Unlike the case of prime numbers where each reduced residue class has infinite number of primes (Dirichlet's Theorem), 
in case of the twin primes, some reduced residue class may not have more than one prime pair. This is because the existence 
of the second prime adds extra 
constraint on the first one. For example, there is no twin prime pair in the residue class $a=1$ (mod 6). For $q=10$, the 
residue class $a=3$ has only one prime pair, namely 3 and 5. These observations can be concisely presented in the following theorem.
\begin{theorem}
The reduced residue class $a$ (mod q) can not have more than one twin prime pair when $(a+2,q)>1$. For such $a$, 
if both $a$ and $a+2$ are prime numbers, then they represent the only twin prime pair in the residue class $a$.
\label{th3.1}
\end{theorem}

This theorem can simply be proved by noting that for any prime $p\equiv a$ (mod $q$), $p+2$ will be forced to be a composite
if $(a+2,q)\neq 1$. Here only exception may happen when the number $a+2$ is itself a prime number. In such a case, if $a$ is also a 
prime number, then the pair consisting of $a$ and $a+2$ is the only twin prime pair in the class $a$.

It may be mentioned in the passing that some non-reduced residue class may have one twin prime pair. For example, when $q=9$, residue class 
$a=3$ has one prime pair, namely 3 and 5. 

Theorem \ref{th3.1} helps us to come up with a conjecture for twin primes in the line of the Dirichlet's Theorem for prime numbers. 
For that we first define below an admissible class in the present context of twin prime distribution.
\begin{definition}
A residue class $a$ (mod $q$) is said to be admissible if $(a,q)=1$ and $(a+2,q)=1$.
\label{df3.1}
\end{definition}

With this definition of admissible class, we now propose the following conjecture. 
\begin{conjecture}
Each admissible residue class of a given modulus contains infinite number of twin prime pairs. 
\label{cj3.1}
\end{conjecture}

A stronger version of this conjecture can be formulated with the definition of a {\it modified totient function} $\phi_2(q)$. 
This function gives the number of {\it admissible} residue classes modulo $q$.
\begin{definition}
The function $\phi_2(q)$ denotes the number of admissible residue classes modulo $q$, i.e., $\phi_2(q)=\#\{a<q~|(a,q)=1 ~{\rm and}~ (a+2,q)=1\}$.
\label{df3.2}
\end{definition}

Clearly $\phi_2(q)\le \phi(q)$, where $\phi(q)$ is the Euler's totient function. While calculating $\phi_2(q)$, a reduce residue class $a$ is
excluded if there exists a prime divisor ($p$) of $q$ so that $p|a+2$. This observation helps
us to deduce that $\phi_2(2^n)=\phi(2^n)=2^{n-1}$ for $n\ge 1$, and for any odd prime $p$, $\phi_2(p^m)=\phi(p)-p^{m-1}=p^m(1-2/p)$ when $m\ge 1$. 
In Section \ref{sec7} we prove that 
\begin{equation}
\begin{split}
\phi_2(q) &= q \prod_{p|q}(1-\frac 2 p) \text{~~for odd q, and}\\ \nonumber 
\phi_2(q) &= q (1-\frac 1 2) \prod_{p>2,~p|q}(1-\frac 2 p) \text{~~for even q}.
\end{split}
\end{equation}
If we take $\phi_2(1) = 1$ then the function $\phi_2(q)$ is multiplicative. Above formula for $\phi_2(q)$ is stated as a theorem in Section \ref{sec4}. 

The function $\phi_2(q)$ can also be expressed in terms of the principal Dirichlet character ($\chi_1$) in the following way: 
$\phi_2(q)=\sum_{i=1}^q\chi_1(i)\chi_1(i+2)=\sum_{k=1}^{\phi(q)}\chi_1(a_k+2)$, where the second sum is over the reduced residues ($a_k$).
The values of the function for some integers are as follows: $\phi_2(4)= \phi(4)=2$, 
$\phi_2(6)= \phi(6)-1=1$, $\phi_2(7)= \phi(7)-1=5$ and $\phi_2(9)= \phi(9)-3=3$. 

Using the function $\phi_2(q)$, we now propose the following conjecture regarding the twin prime counting function for a given admissible class.
\begin{conjecture}
For a given $q$ and an admissible class $a$ (mod $q$), we have  
\begin{equation}
{\displaystyle \lim_{x\rightarrow\infty}} \frac{\pi_2(x)}{\pi_2(x;q;a)} = \phi_2(q). \nonumber
\end{equation}
\label{cj3.2}
\end{conjecture}

We end this section with a conjecture on the arithmetic progression of twin primes in line with the Green-Tao theorem \cite{greentao08}. 
For some positive integers $a$ and $b$, let the sequence $a+kb$ represents (first members of) twin primes for $k=0, 1, 2, \cdots, l-1$ ($l>2$). 
This sequence is said to represent an arithmetic progression of twin primes with $l$ terms. 
\begin{conjecture}
For every natural number $l$ ($l>2$), there exist arithmetic progressions of twin primes with $l$ terms.
\label{cj3.3}
\end{conjecture}

Some examples for the arithmetic progressions of twin primes are as follows: 41+420$\cdot$k, for k = 0, 1, .., 5; 
51341+16590$\cdot$k, for k = 0, 1, .., 6; 2823809+570570$\cdot$k, 
for k = 0, 1, .., 6. Each number in a progression represents the first member of a twin prime.

\section{Functions $\mu_2(n)$ and $\phi_2(n)$}
\label{sec4}

We begin this section with the following definition of a {\it modified M\"{o}bius function} $\mu_2(n)$. 
\begin{definition}
If $\omega_o(n)$ denotes the number of odd prime divisors of $n$, then the {\it modified M\"{o}bius function} is defined as
\begin{equation}
\mu_2(n)=\mu(n)2^{\omega_o(n)},\nonumber 
\end{equation}
where $\mu(n)$ is the M\"{o}bius function.
\label{df4.1}
\end{definition}
The function $\mu_2(n)$ is multiplicative. The values of $\omega_o(n)$ for some $n$ are as follows: $\omega_o(1)=0$, $\omega_o(2)=0$, $\omega_o(3)=1$ 
and $\omega_o(6)=1$.

The main reason we defined $\mu_2(n)$ is that it is possible to express $\phi_2(n)$ in terms of this modified M\"{o}bius function. The following 
theorem (\ref{th4.1}) gives two different expressions of $\phi_2(n)$. The proof of the theorem is given in Section \ref{sec7}.

\begin{theorem}
The function $\phi_2(n)$, as defined in Definition \ref{df3.2}, takes the following product form 
\begin{equation}
\phi_2(n) = n (1-\frac{\theta_n}{2}) \prod_{p>2,~p|n}(1-\frac 2 p), 
\label{eq_ph2}
\end{equation}
where $\theta_n$ is either 0 or 1 depending on whether $n$ is odd or even respectively. Furthermore, using $\mu_2(n)$ function, 
one can express $\phi_2(n)$ as the following divisor sum:
\begin{equation}
\phi_2(n) = n \sum_{d|n} \frac{\mu_2(d)}{d}. 
\label{eq_p2m2}
\end{equation}
\label{th4.1}
\end{theorem}

In the following we discuss some properties of the functions $\mu_2(n)$ and $\phi_2(n)$.

\begin{lemma}
The divisor sum of the function $\mu_2$ is given by 
\begin{equation}
\sum_{d|n}\mu_2(d)= \left\{\begin{array}{ll}
0 & if~~ n~is~ even \\ \nonumber
(-1)^{\omega(n)} & if~~ n~is~ odd.
\end{array}\right. 
\end{equation}
\label{lm4.1}
\end{lemma}
\begin{proof}
The formula is true for $n=1$ since $\mu_2(1)=\mu(1)2^{0} = 1$. Assume now that $n>1$ and is written $n=p_1^{a_1}p_2^{a_2}\cdots p_k^{a_k}$. First we take 
$n$ to be even with $p_1 = 2$ and $a_1\ge1$. We note that for any odd divisor $d$, there is an even divisor $2d$ so that $\mu_2(d)+\mu_2(2d) = 0$. It may 
be noted that we have $\mu_2(2^a d)=0$ for $a>1$. This shows that $\sum_{d|n}\mu_2(d) = 0$ when $n$ is even.

We now take $n$ to be odd. Using Lemma \ref{lm8.1}, we get the following.
\begin{equation}
\sum_{d|n}\mu_2(d) = \sum_{d|n} \mu(d) 2^{\omega_o(d)} = \prod_{p|d}(1-2)=(-1)^{\omega(n)}. \nonumber
\end{equation}
\end{proof}

\begin{lemma}
If $n=\prod_{p|n}p^{a_p}$, then the divisor sum of the function $\phi_2$ is given by
\begin{equation}
\sum_{d|n}\phi_2(d) = n \prod_{p>2,~p|n}\frac{\phi_2(p)+p^{-a_p}}{\phi(p)}.  \nonumber
\end{equation}
\label{lm4.2}
\end{lemma}
\begin{proof}
For odd $n$, taking advantage of the multiplicative property of $\phi_2$, we write 
$\sum_{d|n}\phi_2(d) = \prod_{p|n}\{1+\phi_2(p)+\phi_2(p^2)+\cdots+\phi_2(p^{a_p})\}$. Now since $\phi_2(p^m)=p^m(1-\frac 2p)$ for odd $p$ and $m\ge1$, we 
get $\sum_{d|n}\phi_2(d) = \prod_{p|n}\frac{p^{a_p}(p-2)+1}{p-1} = n \prod_{p|n}\frac{\phi_2(p)+p^{-a_p}}{\phi(p)}$. For even $n$, there will be an 
additional multiplicative factor to this result; the factor is $(1+\phi_2(2)+\phi_2(2^2)+\cdots+\phi_2(2^{a_2})) = 2^{a_2}$. This factor is a 
part of $n$.
\end{proof}

\begin{theorem}
Let $D(\mu_2,s)$ denote the Dirichlet series for $\mu_2$, i.e., $D(\mu_2,s):=\sum_{n=1}^{\infty}\frac{\mu_2(n)}{n^s}$. We have
\begin{equation}
D(\mu_2,s) = (1-\frac{1}{2^s})\prod_{p>2}(1-\frac{2}{p^s}) {~~ for ~~} Re(s) >1.\nonumber
\end{equation}
\label{th4.2}
\end{theorem}
\begin{proof}
First we note that $2^{\omega_o(n)}\le 2^{\omega(n)}\le d(n)$, where $d(n)$ counts the number of divisors of $n$. It is known that 
$\sum_{n=1}^{\infty}\frac{d(n)}{n^s}$ is an absolutely convergent series for $Re(s)>1$ (chapter 11, \cite{apostol76}); this implies that the series
$\sum_{n=1}^{\infty}\frac{\mu_2(n)}{n^s}$ also converges absolutely for $Re(s)>1$.

Using Lemma \ref{lm8.2}, we have $D(\mu_2,s)=\prod_p\{1+\frac{\mu_2(p)}{p^s}\} = (1-\frac{1}{2^s})\prod_{p>2}(1-\frac{2}{p^s})$.
\end{proof}

\begin{theorem}
Let $D(\phi_2,s)$ denote the Dirichlet series for $\phi_2$, i.e., $D(\phi_2,s):=\sum_{n=1}^{\infty}\frac{\phi_2(n)}{n^s}$. We have
\begin{equation}
D(\phi_2,s)= \zeta(s-1) D(\mu_2,s) \text{~~ for ~~} Re(s) >2,\nonumber
\end{equation}
where $\zeta(s-1)$ is the Riemann zeta function.
\label{th4.3}
\end{theorem}
\begin{proof}
We first note that $\phi_2(n)\le n$. So the series $\sum_{n=1}^{\infty}\frac{\phi_2(n)}{n^s}$ is absolutely convergent for $Re(s)>2$.
The theorem now can be simply proved by taking $f(n)=\mu_2(n)$ and $g(n)=n$ in the Theorem \ref{th8.1}.
\end{proof}

\noindent {\bf Remark.} We here note that $\frac{1}{D(\mu_2,s)}= (1-\frac{1}{2^s})^{-1}\prod_{p>2}(1-\frac{2}{p^s})^{-1} = 
(1+\frac{1}{2^s}+\frac{1}{2^{2s}}+\cdots)\prod_{p>2}(1+\frac{2}{p^s}+\frac{2^2}{p^{2s}}+\cdots) = \sum_{n=1}^{\infty}\frac{2^{\Omega_o(n)}}{n^s}$,
where $\Omega_o(n)$ counts the number of odd prime divisors of $n$ with multiplicity. Now in analogy with the relation 
$\left(\sum_{n=1}^{\infty}\frac{\mu(n)}{n^s}\right)^{-1} = \zeta(s)$, one may define $\zeta_2(s):=\sum_{n=1}^{\infty}\frac{2^{\Omega_o(n)}}{n^s}$ so 
that Theorems \ref{th4.2} and \ref{th4.3} can be written respectively as $D(\mu_2,s) = 1/\zeta_2(s)$ and $D(\phi_2,s)=\zeta(s-1)/\zeta_2(s)$.

\begin{theorem}
We have the following weighted average of $\mu_2(n)$:
\begin{equation}
\sum_{n\le x} \mu_2(n)\left[\frac{x}{n}\right] = L(x) + L(\frac x2) + 2L(\frac x4) + 4L(\frac x8) + \cdots, \nonumber
\end{equation}
where $L(x) = \sum_{n\le x} (-1)^{\omega(n)}$. It is understood that $L(x)=0$ for $0 < x <1$.
\label{th4.4}
\end{theorem}
\begin{proof}
Using Theorem \ref{th8.3} with $f(n) = \mu_2(n)$ and $g(n) = 1$, we get
\begin{align*}
\sum_{n\le x} \mu_2(n)\left[\frac{x}{n}\right] = & \sum_{n\le x} \sum_{d|n} \mu_2(d) & \\ 
 = & \sum_{odd~ n \le x} (-1)^{\omega(n)} & \text{(using Lemma \ref{lm4.1})}\\
 = & L(x) - \sum_{n \le x/2} (-1)^{\omega(2n)} & \text{(where $L(x) = \sum_{n \le x} (-1)^{\omega(n)}$)}\\
\end{align*}
Since $\sum_{n \le x/2} (-1)^{\omega(2n)} = \sum_{odd~ n \le x/2} (-1)^{\omega(2n)}+\sum_{n \le x/4} (-1)^{\omega(4n)}$ and $\omega(2n) = \omega(n) + 1$ for odd $n$,
we get $K(x) = L(x) + K(\frac x2) - \sum_{n \le x/4} (-1)^{\omega(4n)}$, where $K(x)=\sum_{n\le x} \mu_2(n)\left[\frac{x}{n}\right]$. Now manipulating the last 
term repeatedly we get the following series: $K(x) = L(x) + K(\frac x2) + K(\frac x4) + K(\frac x8) + \cdots$. This series can be further manipulated to find $K(x)$
in the following way: $K(x) = L(x) + L(\frac x2) + 2L(\frac x4) + 4L(\frac x8) + \cdots$.
\end{proof}

\section{Function $\phi_2(n)$ and twin primes}
\label{sec5}

We noted in Conjecture \ref{cj3.2} how $\phi_2(n)$ can be helpful in studying the twin prime distributions. In this section we will present some more
observations on how $\phi_2(n)$ can be useful in the study of twin primes. First we mention a simple application of $\phi_2(n)$ in the form of the following 
theorem.
\begin{theorem}
For a given $l \ge 1$, consider a set $\mathcal{P}$ of any $l$ distinct primes. There always exists a number $m>1$ so that $(p,m(m+2))=1$ for every prime 
$p \in \mathcal{P}$.
\label{th5.1}
\end{theorem}
\begin{proof}
When $l=1$ and $p\ge 5$, from Theorem \ref{th4.1}, we get $\phi_2(p) \ge 2$. This confirms the existence of at least one $m>1$ so that $(p,m(m+2))=1$. It is also 
easy to explicitly check that the theorem is correct when $p=2$ or 3.\\
When $l=2$, and two primes are 2 and 3, we can explicitly check that the theorem is valid. For example, $m=5$ works here. When one or both primes are 2 or 3, it 
is easy to see that $\phi_2(p_1p_2)\ge 2$ for two primes $p_1$ and $p_2$.\\
When $l\ge 3$, we can define $P_l = {\displaystyle \prod_{p \in \mathcal{P}}}p$, and note that $\phi_2(P_l) \ge 2$. Which confirms the existence of at least 
one $m>1$ so that $(p,m(m+2))=1$ for every prime $p \in \mathcal{P}$.
\end{proof}

\begin{corollary}
For a given $n>2$, there always exists a number $m>n$ so that $(p,m(m+2))=1$ for every prime $p\le n$.
\label{cor5.1}
\end{corollary}
\begin{proof}
This corollary can be shown to be a consequence of Theorem \ref{th5.1}. We will give here a separate direct proof of the corollary.
Let $P_n = {\displaystyle \prod_{p\le n}}p$. From Theorem \ref{th4.1}, we have $\phi_2(P_n)={\displaystyle \prod_{odd~p\le n}}(p-2)\ge 1$. We recall that, 
between 1 and $P_n$, there are $\phi_2(P_n)$ twin coprime pairs to $P_n$. In this case all these pairs (one pair when $2<n<5$) lie above $n$. 
This proves the theorem. 
\end{proof}

This theorem tells us that it is always possible to find a twin coprime pair to the primorial ($P_n$) upto a given number ($n>2$). 
It may be mentioned here that we can have a stronger statement than the Corollary \ref{th5.1} by noting that there are $\phi_2(P_n)$ number 
of twin coprime pairs in ($n,P_n$). It is also interesting to note that the Theorem \ref{th5.1} implies that there are infinite number of primes.

According to the Hardy-Littlewood conjecture, $\pi_2(x) \sim 2C_2\frac{x}{(ln x)^2}$, where  $C_2$ is the so-called twin prime constant; this constant is 
given by the following product form: $C_2= \prod_{p>2}\frac{p(p-2)}{(p-1)^2}$. In terms of $\phi(n)$ and $\phi_2(n)$, this constant can be written in the 
following series forms,  

\begin{equation}
C_2 =  \sum_{odd~n=1}^{\infty}\frac{\mu(n)}{\phi^2(n)},~~\text{and}
\label{eq_c2}
\end{equation}
\begin{equation}
\frac{1}{C_2} =  \sum_{odd~n=1}^{\infty} \frac{\mu^2(n)}{n\phi_2(n)}.
\label{eq_inc2}
\end{equation}
Both the equations can be proved by expressing their right hand sides as Euler products (see also Lemma \ref{lm8.2}; to use it in the present context, 
leave out the sum over the even numbers from the left side and drop $p=2$ from the right side of the equation). 
The Equation \ref{eq_c2} also appears in \cite{golomb60}.

Let $P_n$ denote the product of the first $n$ primes, i.e., 
$P_n=\Pi_{k=1}^{n}p_k$. It is easy to check that 
\begin{equation}
2C_2 = {\displaystyle \lim_{n\rightarrow\infty}} \frac{P_n \phi_2(P_n)}{\phi^2(P_n)} = 
{\displaystyle \lim_{n\rightarrow\infty}} \frac{\phi_2(P^2_n)}{\phi^2(P_n)}.
\label{eq_c2ro}
\end{equation}

The last expression of $C_2$, as appears in Equation \ref{eq_c2ro}, is especially interesting here. It helps us to view the Hardy-Littlewood twin prime 
conjecture from a different angle. We explain this in the following.
\vspace{0.1cm}
\paragraph{\bf Heuristics for twin prime conjecture.} 
We take $X = P_x = \prod_{p\le x}p$ for large $x$. The number of integers which are coprime to and smaller than $X$ is $\phi(X)$. 
It may be mentioned here that all the coprimes, except 1, lie in $[x,X]$. The total number of the coprime pairs of all possible gaps is 
$\frac 12 \phi(X) ( \phi(X)-1) \sim \frac 12 \phi^2(X)$. Among these pairs, the number of the twin coprime pairs (pairs with gap = 2) 
is $\phi_2(X)$. All these twin coprime pairs lie in $[x,X]$. So in $[x,X]$, the fraction of the coprime pairs which 
are the twin coprime pairs is $\frac{2\phi_2(X)}{\phi^2(X)}$. Now the number of primes in $[x,X]$ is $\pi(X)-\pi(x)\sim \pi(X)$.
The number of the all possible prime pairs (with any non-zero gap) is 
$\frac 12 \pi(X)(\pi(X)-1) \sim \frac 12 \pi^2(X)$. Among these prime pairs, the number of the twin prime pairs in 
$[x,X]$ would then be: $\pi_2(X) -\pi_2(x) \sim \frac 12 \pi^2(X) \times \frac{\phi_2(X)}{\frac 12 \phi^2(X)}$. From the 
prime number theorem we know $\pi(X) \sim \frac{X}{\ln X}$. So for large $X$, we finally get 
$\pi_2(X) \sim  \frac{X^2}{\ln^2 X} \times \frac{\phi_2(X)}{\phi^2(X)} = \frac{X}{\ln^2 X} \times \frac{X\phi_2(X)}{\phi^2(X)}$. Now $\frac{X\phi_2(X)}{\phi^2(X)}$ 
converges to $2C_2$ as noted in Equation \ref{eq_c2ro}; this implies: $\pi_2(X) \sim 2C_2 \frac{X}{\ln^2 X}$.

Above argument can also be given in the following ways. We note that $\frac {\pi^2(X)}{\phi^2(X)}$ denotes the fraction of the coprime pairs which are the 
prime pairs (of any possible gaps). 
Now there are $\phi_2(X)$ twin coprime pairs; among them, the number of twin prime pairs would then be 
$\frac {\pi^2(X)}{\phi^2(X)} \times \phi_2(X) \sim 2C_2 \frac{X}{\ln^2 X}$. We also note that $\frac {\pi(X)}{\phi(X)}$ is the probability of a coprime to be 
a prime. So $\left(\frac {\pi(X)}{\phi(X)}\right)^2$ is the probability that the two members of a coprime pair are both primes. This suggest that the number of 
the twin prime pairs among $\phi_2(X)$ twin coprime pairs is $\left(\frac {\pi(X)}{\phi(X)}\right)^2 \times \phi_2(X) \sim 2C_2 \frac{X}{\ln^2 X}$.

In the preceding heuristics, the reason we take $X$ to be a large primorial (product of primes upto some large number) is that the density of coprimes vanishes
for this special case ($\phi(X)/X \rightarrow 0$ as $X \rightarrow \infty$; see Theorem \ref{th5.2}). This vanishing density is a desired property in our 
heuristics as we know that the density of primes has this property ($\pi(N)/N \rightarrow 0$ as $N \rightarrow \infty$). If we take any arbitrary form of $X$ 
then the density of coprimes may not vanish. We may also note here that, as $X \rightarrow \infty$, $\frac{X\phi_2(X)}{\phi^2(X)}$ goes to a limit when the special 
form of $X$ is taken. On the other hand if $N$ increases arbitrarily then the function $\frac{N\phi_2(N)}{\phi^2(N)}$ does not have any limit.

\vspace{0.4cm}
Next we discuss the estimated values of $\phi(X)/X$ and $\phi_2(X)/X$ when $X$ is the product of primes upto $x$.    
\begin{theorem}[Mertens' result]
As $x\rightarrow \infty$, we have 
${\displaystyle \prod_{p\le x}(1-\frac{1}{p}) = \frac{e^{-\gamma}}{\ln x} + O(\frac{1}{\ln^2 x})}$. Here $\gamma \approx 0.577216$ is the Euler-Mascheroni constant.
\label{th5.2}
\end{theorem}
The proof of this result can be found in the most standard text books.

\begin{theorem}
As $x\rightarrow \infty$, we have ${\displaystyle (1- \frac 12) \prod_{2<p\le x}(1-\frac{2}{p}) = 2C_2\frac{e^{-2\gamma}}{\ln^2 x} + O(\frac{1}{\ln^3 x})}$. Here 
$C_2$ is the twin prime constant and $\gamma$ is the Euler-Mascheroni constant. 
\label{th5.3}
\end{theorem}
\begin{proof}
It is possible to prove Theorem \ref{th5.3} directly using the result from Theorem \ref{th5.2}. It can be done in the following way.\\
${\displaystyle (1- \frac 12) \prod_{2<p\le x}(1-\frac{2}{p}) = \left(2\prod_{2<p\le x} \frac{1-\frac 2p}{(1- \frac 1p)^2}\right) \prod_{p\le x}(1- \frac 1p)^2}$. \\
Now as $x\rightarrow \infty$, ${\displaystyle 2\prod_{2<p\le x} \frac{1-\frac 2p}{(1- \frac 1p)^2} = 2C_2}$ and the last part can be estimated from Mertens' second
result, ${\displaystyle \prod_{p\le x}(1- \frac 1p)^2 = \frac{e^{-2\gamma}}{\ln^2 x} + O(\frac{1}{\ln^3 x})}$. This finally gives us\\
${\displaystyle (1- \frac 12) \prod_{2<p\le x}(1-\frac{2}{p}) = 2C_2\frac{e^{-2\gamma}}{\ln^2 x} + O(\frac{1}{\ln^3 x})}$. 
\end{proof}

\vspace{0.4cm}
The function $\phi_2(n)$ also gives a non-trivial upper bound for $\pi_2(n)$. We make a formal statement in terms of the following theorem.
\begin{theorem}
For $n>2$, $\pi_2(n)\le \phi_2(n)+\omega(n)$.
\label{th5.4}
\end{theorem}
By taking $n$ to be a large primorial (product of primes upto about $\ln n$), we see that for $n \rightarrow \infty$, using Theorem \ref{th5.3}, 
$\pi_2(n) \le \phi_2(n)+\omega(n) \sim \frac {2C_2 n e^{-2\gamma}}{(\ln \ln n)^2}$. This upper bound, of-course, is not very tight; a much better bound can be 
found directly from $\pi(n)$. We note that, except 5, no other prime number is part of two twin prime pairs; accordingly: 
$\pi_2(n)\le \frac 12 \pi(n) \sim \frac{n}{2 \ln n}$.

\section{Eratosthenes sieve and Legendre's formula for twin primes} 
\label{sec6}

The Eratosthenes sieve gives us an algorithm to find the prime numbers in a list of integers. The Legendre's formula uses the algorithm to write the prime-counting
function in the following mathematical form: $\pi(x) - \pi(\sqrt{x})+1 = \sum_{d|P(\sqrt{x})} \mu(d) \left[\frac{x}{d}\right]$, where $P(z) = \prod_{p\le z}p$. Here
we first discuss a similar sieve to detect the twin primes in $[1,x]$. This sieve may be called the Eratosthenes sieve for the twin primes. We use this sieve then 
to write a mathematical formula for twin-prime counting function - which may be called the Legendre's formula for twin primes. 
\vspace{0.1cm}
\paragraph{\bf Eratosthenes sieve for twin primes.}
Assume that we are to find out the (first members) of twin primes in $[1,N]$ with $N>2$. To start the sieving process, we first take a list of integers from 2 to 
$N+2$. We then carry out the following steps. Take the lowest integer ($n$) available in the list (initially $n=2$), and then remove (slashes ``/" in the example 
below) the integers from the list which are multiples of but greater than $n$. Then remove (strike-outs ``---" in the example below) the 
integers from the list which are two less than the integers removed in the last step. Circle the lowest integer available in the list (3 in the first instance). 
Now repeat the process with the circled number (new $n$) and continue the process till we reach $N$ in the list. 
At the end of the process, the circled numbers are the (first members) of the twin primes in $[1,N]$.

\vspace{0.1cm}
\noindent {\bf Example.} To find the (first members) of twin primes below 30, we take the following list of numbers from 2 to 32. We then carry out the above 
mentioned steps to circle out the twin primes below 30.

\noindent \sout{2}  \circled{$\left. 3\right.$}  \sout{$\slashed{4}$}  \circled{$\left. 5\right.$}  \sout{$\slashed{6}$}  \sout{7}  \sout{$\slashed{8}$}  
$\slashed{9}$  \sout{$\slashed{10}$}  
\circled{11} \sout{$\slashed{12}$}  \sout{13}  \sout{$\slashed{14}$}  $\slashed{15}$  \sout{$\slashed{16}$}  \circled{17}  \sout{$\slashed{18}$}  \sout{19}  
\sout{$\slashed{20}$}  $\slashed{21}$  \sout{$\slashed{22}$}  \sout{23}  \sout{$\slashed{24}$}  \sout{$\slashed{25}$}  \sout{$\slashed{26}$}  $\slashed{27}$  
\sout{$\slashed{28}$}  \circled{29}  \sout{$\slashed{30}$}  31  $\slashed{32}$

\vspace{0.1cm}
\noindent So the twin primes (first members) below 30 are 3, 5, 11, 17 and 29.

\vspace{0.1cm}
\paragraph{\bf Legendre's formula for twin primes.}
The main idea behind the sieve above can be used to find a formula for $\pi_2(x)$. We write this formula as a theorem below. The proof of the theorem is give in
Section \ref{sec7}.
\begin{theorem}
If $P(z) = {\displaystyle \prod_{p\le z}}p$ and $[x]$ denotes the greatest integer $\le x$, we have 
\begin{equation}
\begin{split}
\pi_2(x) - \pi_2(\sqrt{x}) & = \sum_{ab|P(\sqrt{x})}\mu(ab) \left[\frac{x-l_{a,b}}{ab}\right], 
\end{split}
\label{eq6.1}
\end{equation}
where $a$ takes only odd values and $b$ takes both odd and even values. 
Here $l_{a,b}$ is the lowest (positive) integer which is divisible by $a$ and simultaneously $l_{a,b}+2$ is divisible by $b$. If we consider $l_{a,b}=at_{a,b}$,
then $t_{a,b}$ is the solution of the congruence $a t_{a,b}+2\equiv 0~(mod~b)$ with $0< t_{a,b} < b$ when $b>1$. For $b=1$, we take $t_{a,b}=1$. 
\label{th6.1}
\end{theorem}

\section{Proofs of Theorems \ref{th4.1} and \ref{th6.1} }
\label{sec7}
We will prove Theorem \ref{th4.1} by two methods - first by the principle of cross-classification (inclusion-exclusion principle) and then by the 
Chinese remainder theorem. The first method is also used for proving Theorem \ref{th6.1}.

\vspace{0.1cm}
\noindent {\bf Proof of Theorem \ref{th4.1}}\\
{\it Method 1:} We first assume that $n$ is an odd number and $p_1, \cdots, p_r$ are the distinct prime divisors of $n$. 
Let $S = \{1,2,\cdots,n\}$ and $S_k = \{m\le n : p_k|m(m+2)\}$. 
It is clear that
\begin{equation}
\phi_2(n) = N\left(S -\bigcup_{i=1}^r S_i\right), \nonumber
\end{equation}
where $N(T)$ denotes the number of elements in a subset $T$ of $S$. Now by the principle of cross-classification (see Theorem \ref{th8.2} for notations and 
formal statement), we have
\begin{equation}
\phi_2(n) = N(S) - \sum_{1\le i \le r}N(S_i) + \sum_{1\le i<j \le r}N(S_iS_j)- \cdots + (-1)^rN(S_1S_2\cdots S_r). \nonumber
\end{equation}
We have $N(S)=n$ and $N(S_i)=2\frac{n}{p_i}$. 
In general, to calculate $N(S_{k_1}\cdots S_{k_l})$, we write $N(S_{k_1}\cdots S_{k_l}) = {\displaystyle \sum_{d|P_l}} N(A_{P_l}^{d})$ where $P_l$ is the product 
of given $l$ primes corresponding to the sets $S_{k_1}, S_{k_2}, \cdots, S_{k_l}$, and $A_{P_l}^{d}=\{m\le n : d|m ~\text{and}~ d'|m+2\}$ with $dd' = P_l$.
For $d'=1$, $N(A_{P_l}^{d})=\frac {n}{P_l}$. For $d'>1$, $N(A_{P_l}^{d})$ is given by the number of solutions for $a$ in the congruence 
$d a + 2 \equiv 0 ~(\text{mod}~d')$ so that  $0<d a \le n$. This congruence has a 
unique solution modulo $d'$; let $t$ be the solution where $1\le t < d'$. All the solutions for $a$ would then be in the form $t+s d'$ where $s$ is so that 
$0<(t+s d') d \le n$. We have now $s \le \frac {n}{P_l} - \frac{t}{d'}$. Since $\frac {n}{P_l}$ is an integer and $\frac{t}{d'}<1$, the maximum value of $s$ can 
be $\frac {n}{P_l}-1$. On the other hand, the minimum value of $s$ is 0. This gives $N(A_{P_l}^{d})= \frac {n}{P_l}$. Since there are $2^l$ distinct divisors of 
$P_l$, we get $N(S_{k_1}\cdots S_{k_l}) = 2^l \frac {n}{P_l}$. Now using the principle of cross-classification, we get 
$\phi_2(n) = n - 2\sum_{1\le i \le r} \frac{n}{p_i} + 2^2 \sum_{1\le i<j \le r} \frac{n}{p_ip_j} - \cdots (-1)^r 2^r \frac{n}{p_1 \cdots p_r} = 
n \sum_{d|n} \frac{\mu_2(d)}{d}$.

If $n$ is even, we can take $p_1=2$ and correspondingly $S_1=\{m\le n : p_1|m \}$. We now have $N(S_1)=\frac{n}{p_1}$. To find $N(S_1S_{k_1}\cdots S_{k_l})$, where
the sets $S_{k_1}, S_{k_2}, \cdots, S_{k_l}$ correspond to $l$ different odd primes, we write 
$N(S_1S_{k_1}\cdots S_{k_l}) = {\displaystyle \sum_{d|P_l}} N(A_{P_l}^{d})$. Here $P_l$ is the product of given $l$ odd primes, and 
$A_{P_l}^{d}=\{m\le n : d|m ~\text{and}~ 2d'|m+2\}$ with $dd' = P_l$. For $d'=1$, $N(A_{P_l}^{d})=\frac {n}{2P_l}$.
For $d'>1$, $N(A_{P_l}^{d})$ is given by 
the number of solutions for $a$ in the congruence $d a + 2 \equiv 0 ~(\text{mod}~2d')$ so that  $0<da \le n$. Proceeding 
as before, we find that $N(A_{P_l}^{d})= \frac {n}{2P_l}$ and $N(S_1S_{k_1}\cdots S_{k_l}) = 2^l \frac {n}{2P_l}$.
Using these values, we finally get $\phi_2(n)=n \sum_{d|n} \frac{\mu_2(d)}{d}$. We may here note that the distinctiveness of even $n$ is taken care by 
$\mu_2(n)$ which is defined by $\mu_2(n)=\mu(n)2^{\omega_o(n)}$ where $\omega_o(n)$ is the number of odd prime divisors of $n$.

\vspace{0.2cm}
\noindent {\it Method 2:} Let $n=p_1^{a_1}p_2^{a_2}\cdots p_r^{a_r}$. Now consider the system of congruences with $i$-th congruence as 
$x\equiv b_i$ (mod $p_i^{a_i}$), where $(b_i,p_i)=(b_i+2,p_i)=1$ and $0<b_i<p_i^{a_i}$. In the congruence, $b_i$ can take $\phi_2(p_i^a)$ values, where, 
as discussed in Section \ref{sec3}, $\phi_2(p^a) = p^a(1-\frac 2p)$ for an odd prime $p$, and $\phi_2(p^a) = p^a(1-\frac 1p)$ for the even prime $p=2$.  
For a given set $(b_1,b_2,\cdots b_r)$, the system of congruences has a unique solution modulo $n$ according to the Chinese remainder theorem. So for 
all possible values of $(b_1,b_2,\cdots b_r)$, there are $\prod_r \phi_2(p_r^{a_r})$ unique solutions which lie above 0 and below $n$. 

\vspace{0.1cm}
\noindent {\bf Proof of Theorem \ref{th6.1}}\\
To find the number of twin primes in $[\sqrt{x},x]$, we can suitably use the concept behind the Eratosthenes sieve for twin primes discussed in Section \ref{sec6}. 
To translate this concept in the mathematical form, we adopt here the approach which is used in the proof (Method 1) of Theorem \ref{th4.1}. 

As before, we take $p_1=2$ and $S_1=\{m\le x : p_1|m \}$. For any odd prime $p_k$, we take $S_k = \{m\le x : p_k|m(m+2)\}$. 
This sets help us to write
\begin{equation}
\pi_2(x)-\pi_2(\sqrt{x}) = N\left(S -\bigcup_{1\le k \le r} S_k\right), \nonumber
\end{equation}
where $r = \pi(\sqrt{x})$ and $N(T)$ denotes the number of elements in a subset $T$ of $S$. Now by the principle of cross-classification (Theorem \ref{th8.2}), 
we have
\begin{equation}
\begin{split}
\pi_2(x)-\pi_2(\sqrt{x}) = & N(S) - \sum_{1\le i \le r}N(S_i) + \sum_{1\le i<j \le r}N(S_iS_j)- \cdots \\
 & + (-1)^rN(S_1S_2\cdots S_r). 
\end{split}
\label{pi2x}
\end{equation}
We have $N(S)=[x]$, $N(S_1)=[\frac{x}{p_1}]$ for $p_1=2$, and $N(S_i)=2[\frac{x}{p_i}]$ for odd primes.
In general, to find $N(S_{k_1}\cdots S_{k_l})$, where the sets $S_{k_1}, S_{k_2}, \cdots, S_{k_l}$ correspond to $l$ different odd primes,
we write $N(S_{k_1}\cdots S_{k_l}) = {\displaystyle \sum_{d|P_l}} N(A_{P_l}^{d})$ where $P_l$ is the product of
given $l$ odd primes and $A_{P_l}^{d}=\{m\le x : d|m ~\text{and}~ d'|m+2\}$ with $dd' = P_l$. For $d'=1$, $N(A_{P_l}^{d})=[\frac {x}{P_l}]$. For 
$d'>1$, $N(A_{P_l}^{d})$ is given by the number of solutions for $a$ in the congruence $d a + 2 \equiv 0 ~(\text{mod}~d')$ so that  $0<d a \le x$. This 
congruence has a unique solution modulo $d'$; let $t_{d,d'}$ be the solution where $1\le t_{d,d'} < d'$. All the solutions for $a$ would then be in the form 
$t_{d,d'}+s d'$ where $s$ is so that $0<d (t_{d,d'}+s d') \le x$. Since $t_{d,d'} < d'$, $s$ can not be a negative integer. In other way, when $d t_{d,d'}>x$ or 
$s \le \frac {x}{P_l} - \frac{t_{d,d'}}{d'} < 0$, then we do not have any solution for $a$; hence in such a case $N(A_{P_l}^{d})=0$. Since $\frac{t_{d,d'}}{d'}<1$, 
we have $-1<\frac {x}{P_l} - \frac{t_{d,d'}}{d'}$. So in the case when $d t_{d,d'}>x$, we can write  
$N(A_{P_l}^{d})=1+\left[ \frac {x}{P_l} - \frac{t_{d,d'}}{d'} \right]$. When $d t_{d,d'}\le x$ or $\frac {x}{P_l} - \frac{t_{d,d'}}{d'}\ge 0$, then $s$ can take 
any value from 0 to $\left[\frac {x}{P_l} - \frac{t_{d,d'}}{d'}\right]$. 
Therefore, in this case also $N(A_{P_l}^{d})=1+\left[ \frac {x}{P_l} - \frac{t_{d,d'}}{d'} \right]$.

To find $N(S_1S_{k_1}\cdots S_{k_l})$, where
the sets $S_{k_1}, S_{k_2}, \cdots, S_{k_l}$ correspond to $l$ different odd primes, we write 
$N(S_1S_{k_1}\cdots S_{k_l}) = {\displaystyle \sum_{d|P_l}} N(A_{P_l}^{d})$.
Here $P_l$ is the product of given $l$ odd primes, and $A_{P_l}^{d}=\{m\le x : d|m ~\text{and}~ 2d'|m+2\}$ with $dd' = P_l$. 
For $d'=1$, $N(A_{P_l}^{d})=[\frac {x}{2P_l}]$. For $d'>1$, $N(A_{P_l}^{d})$ is given by
the number of solutions for $a$ in the congruence $da + 2 \equiv 0 ~(\text{mod}~2d')$ so that  $0<da \le x$. Let $t_{d,2d'}$
be the unique solution of the congruence, where $0 < t_{d,2d'} < 2d'$. All the solutions for $a$ would then be in the form $t_{d,2d'}+2s d'$
where $s$ is so that $0< d (t_{d,2d'}+2s d') \le x$. Proceeding as for the odd $n$ case, we find that 
$N(A_{P_l}^{d})=1+\left[ \frac {x}{2P_l} - \frac{t_{d,2d'}}{2d'} \right]$.

If we denote $N(S)$ by $N(1)$, $N(S_i)$ by $N(p_i)$, $N(S_iS_j)$ by $N(p_ip_j)$ and so on, we can rewrite Equation \ref{pi2x} in the following way,
\begin{equation}
\pi_2(x)-\pi_2(\sqrt{x}) = \sum_{a|P(\sqrt{x})}\mu(a)N(a),
\label{pi2x2}
\end{equation}
where $P(\sqrt{x})= {\displaystyle \prod_{p\le\sqrt{x}}p}$. Now $N(a)=\sum_{d|a}' N(A_{P_l}^{d})$ where $\sum'$ indicates the sum over only the odd divisors ($d$) 
of $a$ and $N(A_{P_l}^{d})=1+\left[ \frac {x}{a} - \frac{t_{d,d'}}{d'} \right]$ with $d'=a/d$ for $d'>1$. Here $0<t_{d,d'}<d'$ and 
$d t_{d,d'} +2 \equiv 0  ~(\text{mod}~d')$. For $d'=1$, we take $t_{d,d'}=1$. 
Using the result that $\sum_{d|n}\mu(d) = 0$ for $n>1$, we finally get from Equation \ref{pi2x2}, 
$\pi_2(x)-\pi_2(\sqrt{x}) = {\displaystyle \sum_{dd'|P(\sqrt{x})}\mu(dd')\left[ \frac {x}{dd'} - \frac{t_{d,d'}}{d'} \right]}$, where $d$ is always odd and $d'$ 
can take both odd and even values.

\section{Some standard results}
\label{sec8}

In this section we will provide some useful standard results without their proofs.

\begin{theorem}
Consider two functions $F(s)$ and $G(s)$ which are represented by the following two Dirichlet series,
\begin{equation}
\begin{split}
F(s) &= \sum_{n=1}^{\infty}\frac{f(n)}{n^s} ~~for~~ Re(s)>a, ~~and \\ \nonumber
G(s) & = \sum_{n=1}^{\infty}\frac{g(n)}{n^s} ~~for~~ Re(s)>b.
\end{split}
\end{equation}
Then in the half-plane where both the series converge absolutely, we have
\begin{equation}
F(s)G(s)= \sum_{n=1}^{\infty}\frac{h(n)}{n^s}, \nonumber
\end{equation}
where $h=f\ast g$, the Dirichlet convolution of $f$ and $g$:
\begin{equation}
h(n) = \sum_{d|n}f(d)g(\frac nd). \nonumber
\end{equation}
\label{th8.1}
\end{theorem}

The series for $F(s)$ and $G(s)$ are absolutely convergent respectively for $Re(s)>a$ and $Re(s)>b$. The proof of the theorem can be
found in any standard number theory text book (e.g. chapter 11, \cite{apostol76}).

\begin{lemma}
For a multiplicative function $f$, we have
\begin{equation}
\begin{split}
\sum_{d|n}\mu(d)f(d) = &\prod_{p|n}(1-f(p)) ~~\text{and} \\\nonumber
\sum_{d|n}\mu^2(d)f(d) = &\prod_{p|n}(1+f(p)).
\end{split}
\end{equation}
\label{lm8.1}
\end{lemma}

\begin{lemma}
Let $f$ be a multiplicative function so that the series $\sum_n f(n)$ is absolutely convergent. Then the sum of the series
can be expressed as the following absolutely convergent infinite product over all primes,
\begin{equation}
\sum_{n=1}^{\infty} f(n) = \prod_{p} \{1+f(p)+f(p^2)+ \cdots\}. \nonumber
\end{equation}
\label{lm8.2}
\end{lemma}

\begin{theorem}[Principle of cross-classification]
Let $S$ be a non-empty finite set and $N(T)$ denotes the number of elements of any subset $T$ of $S$. If $S_1, S_2, ..., S_n$ are given subsets of $S$, then
\begin{equation}
\begin{split}
N\left(S-\bigcup_{i=1}^r S_i\right) = & ~N(S) - \sum_{1\le i \le n}N(S_i) + \sum_{1\le i<j \le n}N(S_iS_j)\\ \nonumber
 &-\sum_{1\le i<j<k \le n}N(S_iS_jS_k)+ \cdots + (-1)^nN(S_1S_2\cdots S_n),
\end{split}
\end{equation}
where $S-T$ consists of those elements of $S$ which are not in $T$, and $S_iS_j$, $S_iS_jS_k$, etc. denote respectively $S_i\cap S_j$, $S_i\cap S_j\cap S_k$, etc.
\label{th8.2}
\end{theorem}

\begin{theorem}
If $h(n)=\sum_{d|n}f(d)g(\frac nd)$, let\\
\begin{equation}
H(x) = \sum_{n\le x} h(n), ~~ F(x) = \sum_{n\le x} f(n), ~~ \text{and} ~~ G(x) = \sum_{n\le x} g(n). \nonumber
\end{equation}
We then have \\
\begin{equation}
H(x) = \sum_{n\le x} f(n) G(\frac xn) = \sum_{n\le x} g(n)F(\frac xn). \nonumber
\end{equation}
\label{th8.3}
\end{theorem}
The proof of the theorem can be found in, for example, chapter 3 of Ref. \cite{apostol76}. 

\section{Biases in twin prime distribution}
\label{sec10}

Our analysis of the first 500 million prime numbers and corresponding about 30 million twin prime pairs 
($\pi(x_0)=5\times 10^8$ and $\pi_2(x_0)=29,981,546$) shows three different types of biases both in the distributions of the prime 
numbers and the twin prime pairs. For plotting different arithmetic functions, we take a data point after every 50 primes 
($x_i=p_{_{50i}}$) while presenting the results concerning prime numbers, and we take a data point after every 25 twin prime 
pairs ($x_i=\hat{p}_{_{25i}}$) while presenting the results concerning twin primes. In the following we present our findings.

\subsection{Type-I bias} If $\pi_2(x;4;1)$ and $\pi_2(x;4;3)$ represent the number of prime pairs $\le x$ in the residue class $a=1$ 
and $a=3$ respectively, we find that $\pi_2(x;4;1)>\pi_2(x;4;3)$ for the most values of $x$. This can be seen from the Table \ref{tab1}. 
The bias is also evident from Figure \ref{fig1} 
where we plot the functions  $\delta_2(x;4;3) = \frac{\pi_2(x;4;3)}{\pi_2(x)}$ and 
$\bar{\Delta}_2(x;4;3,1) = (\pi_2(x;4;3) - \pi_2(x;4;1))\frac{ln(x)}{10\sqrt{x}}$. In the definition of $\bar{\Delta}_2(x;4;3,1)$, 
the factor 10 in the denominator is an overall scaling factor. Ideally, without bias, $\delta_2(x;4;3)$ should be 0.5 
and $\bar{\Delta}_2(x;4;3,1)$ should be zero. We also plot the corresponding functions for the prime numbers; we see that, while the 
prime numbers are biased towards the residue class $a=3$, the twin prime pairs are in contrast biased towards the residue class $a=1$.

\begin{table}[th]
\renewcommand
\arraystretch{1.5}
\noindent\[
\begin{array}{r|r r r r r}
x & \pi(x;4;3) & \Delta(x;4;3,1) & \pi_2(x;4;3) & \Delta_2(x;4;3,1)\\
\hline
10^{7} & 332398 & 218 & 29498 & 16 \\
5\cdot10^{7} & 1500681 & 229 & 119330 & -441 \\
10^{8} & 2880950 & 446 & 219893 & -526 \\
5\cdot10^{8} & 13179058 & 2250 & 919192 & -1786 \\
10^{9} & 25424042 & 551 & 1711775 & -956 \\
5\cdot10^{9} & 117479241 & 4260 & 7308783 & -600 \\
10^{10} & 227529235 & 5960 & 13706087 & -505 \\
\end{array}
\]
\caption{(Type-I bias) Here biases in case of primes and twin primes are respectively quantified by the functions 
$\Delta(x;4;3,1)=\pi(x;4;3)-\pi(x;4;1)$ and $\Delta_2(x;4;3,1)=\pi_2(x;4;3)-\pi_2(x;4;1)$. } 
\label{tab1}
\end{table}

\begin{figure}[thb]
\includegraphics[width=0.83\linewidth]{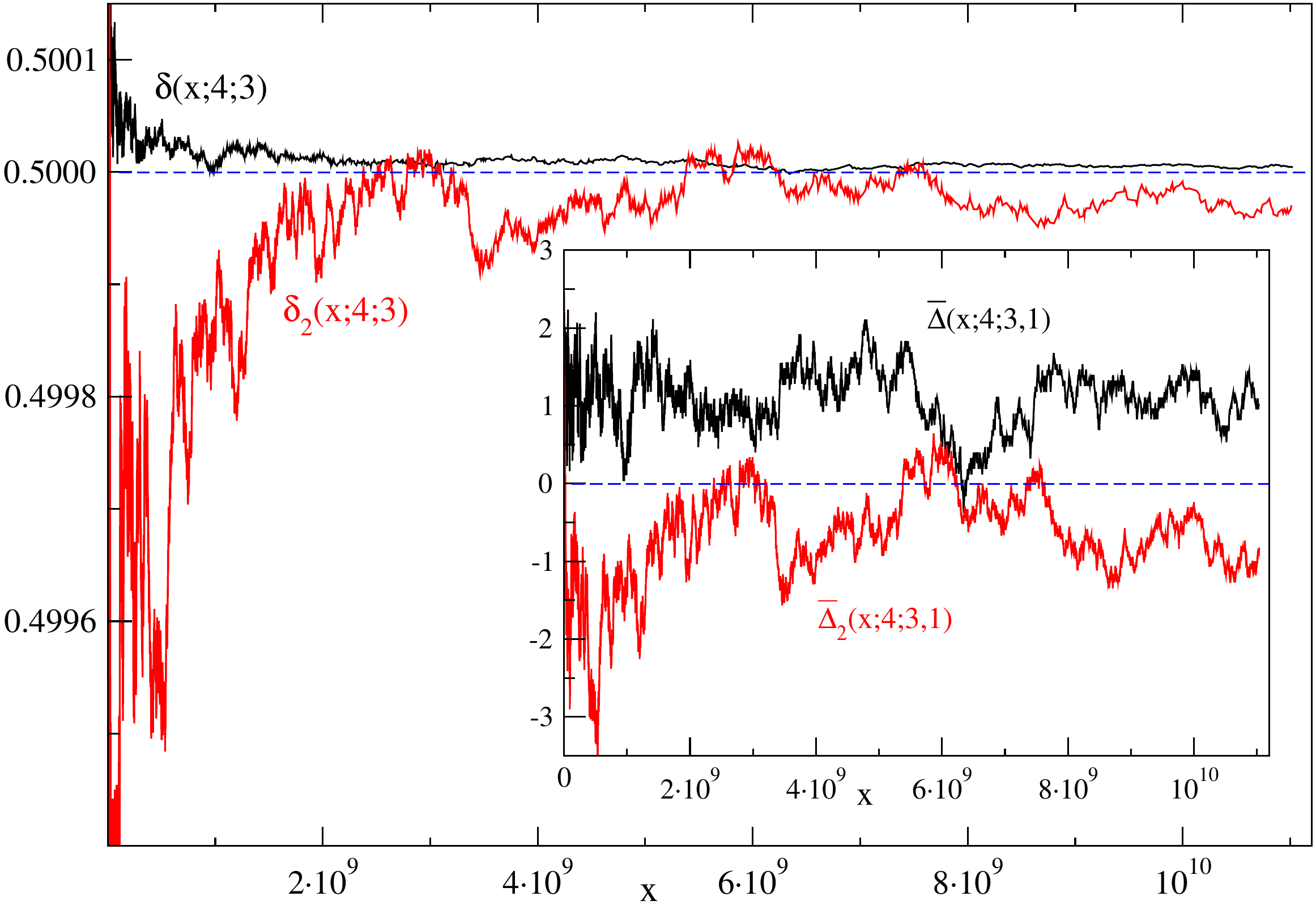}
\caption{(Type-I bias) Plots of the functions $\delta_2(x;4;3)$ and $\bar{\Delta}_2(x;4;3,1)$. The corresponding plots for the prime numbers are also 
shown. The broken lines show expected results if there were no bias.}
\label{fig1}
\end{figure}

Although the overall bias is seen to be towards the residue class $a=1$ for the prime pairs, there is an interval in between where the class $a=3$ is preferred. 
In fact we numerically find that upto about $x\approx50,000$, the residue class $a=1$ is preferred, then upto about $x\approx10^7$, the residue class $a=3$ is 
preferred. After this for a very long interval (we check upto about $x\approx 1.1\times 10^{10}$), the residue class $a=1$ is again preferred for the twin primes.

We also calculate the Brun's constant separately for the two residue classes. Let $B_2(x;4;1) =  (\frac15 +\frac17)+(\frac{1}{17} +\frac{1}{19}) + \cdots$ 
and $B_2(x;4;3) = (\frac13 + \frac15) + (\frac{1}{11}+\frac{1}{13})+\cdots$. Here the first series involves the twin pairs whose first members are congruent 
to 1 (mod 4), similarly, the second series involves the twin primes whose first members are congruent to 3 (mod 4). In both cases we only consider the prime 
pairs in $[1,x]$. We find that $B_2(x_0;4;1)\approx0.802233$ and $B_2(x_0;4;3)\approx0.985735$, where $x_0$ is given by $\pi_2(x_0)=29,981,546$. 
We note here that $B_2(x_0) = B_2(x_0;4;1) + B_2(x_0;4;3) \approx 1.787967$. The value of $B_2(x_0)$ is still somewhat far 
from its known value of $B_2\approx1.902161$; this is due to extremely slow convergence of the series of the reciprocals of twin primes. We find that 
$B_2(x_0;4;3)$ keeps a lead over $B_2(x_0;4;1)$ from the beginning (the first twin pair belongs to class $a=3$). Since $B_2-B_2(x_0)\approx0.114194$ and 
$B_2(x_0;4;3)-B_2(x_0;4;1)=0.183502$, it can be concluded that, irrespective of how large $x$ is, $B_2(x;4;3)$ will always be larger than $B_2(x;4;1)$.
In future study, it will be interesting to find out whether $lim_{x\rightarrow \infty} B_2(x;4;1)$ and 
$lim_{x\rightarrow \infty} B_2(x;4;3)$ exist and, if so, what the corresponding limits are. 

\begin{figure}[tb]
\includegraphics[width=0.80\linewidth]{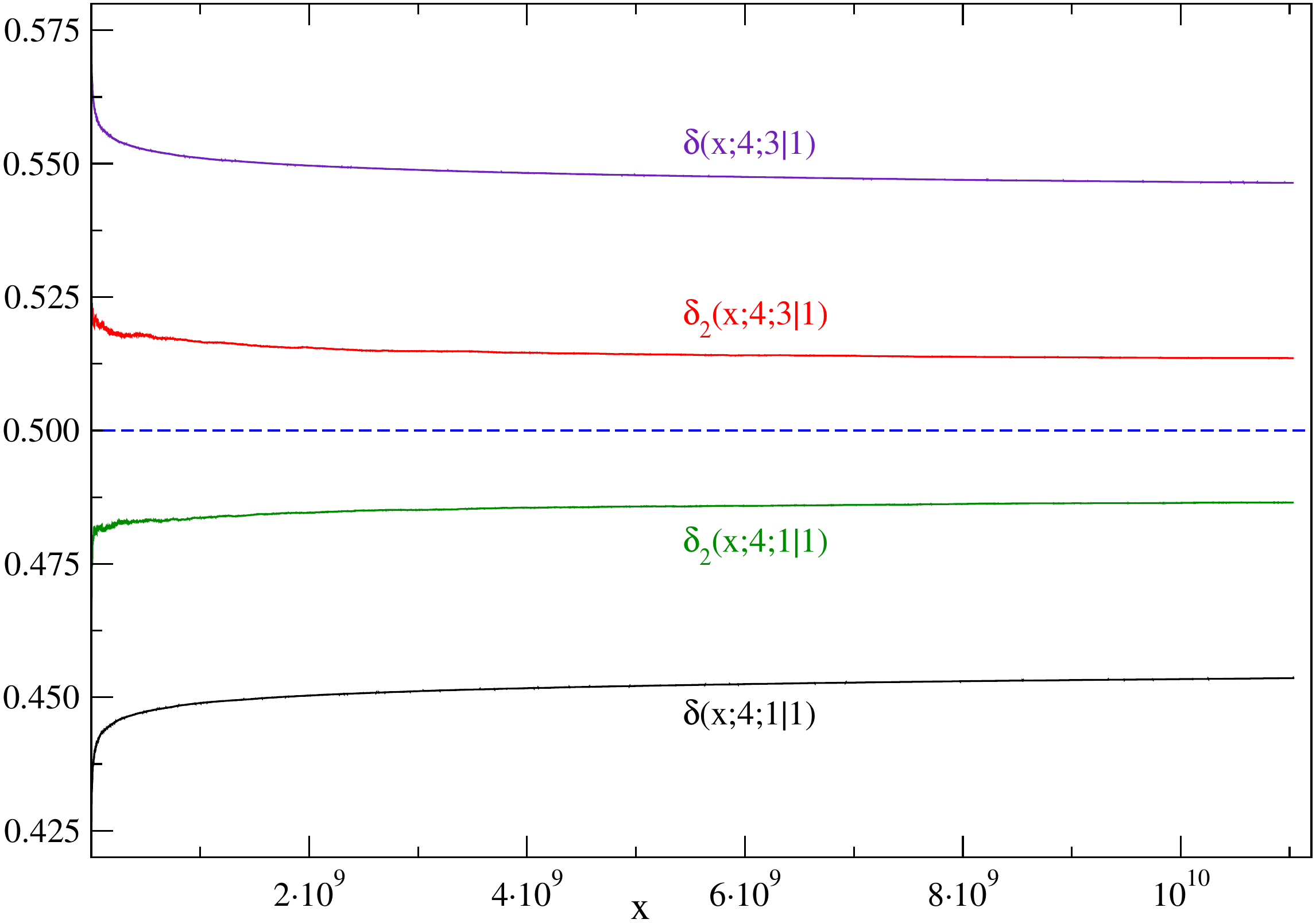}
\caption{(Type-II bias) Plots of the functions $\delta_2(x;4;1|1)$ and $\delta_2(x;4;3|1)$ along with their counterparts for the prime numbers are shown
here. Plots of the complementary functions, like $\delta_2(x;4;3|3) = 1 - \delta_2(x;4;3|1)$, are not provided. 
The broken line shows expected result if there were no bias.}
\label{fig2}
\end{figure}

\subsection{Type-II bias} A second type of bias can be found if we consider the consecutive twin prime pairs. Our numerical results 
show that after a prime pair of certain residue class, it is more probable to find the next prime pair to be from the different residue class. To quantify 
this bias, we define the following functions: $\delta_2(x;4;1|1):= \frac{\pi_2(x;4;1|1)}{\pi_2(x;4;1)}$ and 
$\delta_2(x;4;1|3):= \frac{\pi_2(x;4;1|3)}{\pi_2(x;4,1)}$. Here $\pi_2(x;4;a_i|a_j)$ denotes the number of twin prime pairs ($\le x$) 
belonging to the residue class $a_i$ provided that their next pairs belong to the class $a_j$. 
It may be noted that $\delta_2(x;4;1|1)+\delta_2(x;4;1|3)=1$. 
Functions $\delta_2(x;4;3|1)$ and $\delta_2(x;4;3|3)$ are defined similarly. The plots of function $\delta_2(x;4;1|1)$ and $\delta_2(x;4;3|1)$
can be found in Figure \ref{fig2}. The corresponding plots for the prime numbers are also shown in the figure. If we consider the twin prime 
pairs to be completely uncorrelated, the values of these functions should be 0.5. But we see that, for example, 
$\delta_2(x;4;1|3)>0.5>\delta_2(x;4;1|1)$ for all values of $x$ that we investigated. Exactly same behavior is seen in case of prime numbers, 
although the bias is more here. The bias in consecutive primes was investigated earlier in Ref. \cite{oliver16}.

We here would like to point out that the functions like $\delta_2(x;4;1|1)$ and $\delta_2(x;4;3|1)$ are quite regular and smooth compared  
to a function like $\delta_2(x;4,3)$ or $\bar{\Delta}_2(x;4;3,1)$. The values of the functions $\delta_2(x;4;1|1)$ and $\delta_2(x;4;3|1)$ for 
some values of $x$ can be found in Table \ref{tab2}.  
\begin{table}[]
\renewcommand
\arraystretch{1.5}
\noindent\[
\begin{array}{c|c c c c}
x & \delta(x;4;1|1) & \delta(x;4;3|1) & \delta_2(x;4;1|1) & \delta_2(x;4;3|1)\\
\hline
10^{7} & 0.4350 & 0.5647 & 0.4769  & 0.5228 \\
10^{8} & 0.4434 & 0.5565 & 0.4815  & 0.5198 \\
10^{9} & 0.4489 & 0.5511 & 0.4836 & 0.5166 \\
10^{10} & 0.4534 & 0.5466 & 0.4864 & 0.5136 \\
\end{array}
\]
\caption{(Type-II bias) Variation of functions $\delta(x;4;a_i|a_j)$ and $\delta_2(x;4;a_i|a_j)$ are very slow and smooth. 
Values of the functions are rounded off to the four decimal places.}
\label{tab2}
\end{table}

Although $\delta_2(x;4;a_i|a_j)$ varies very slowly, we expect all these functions to approach 0.5 as $x$ goes to infinity. Assuming that there are 
infinite number of twin prime pairs, we propose the following conjecture.
\begin{conjecture}
For any two residue classes $a_i$ and $a_j$, we have $\delta_2(x;4;a_i|a_j)=0.5$ as $x\rightarrow\infty$.
\label{cj10.1}
\end{conjecture}

\subsection{Type-III bias} A third type of bias can be found if we study the gap between the first members of consecutive twin prime pairs. To
quantify the bias, we define the following functions: $\delta^{+}_2(x):=\frac{\pi^{+}_2(x)}{\pi_2(x)}$, where $\pi^{+}_2(x)$ is the number
of prime pairs ($\le x$) for each of which the following twin prime gap plus 1 is a prime number, i.e., 
$\pi^{+}_2(x) = \#\{\hat{p}_n \le x~| (\hat{p}_{n+1}-\hat{p}_n + 1) ~\text{is a prime} \}$.
\begin{table}[ht]
\renewcommand
\arraystretch{1.5}
\noindent\[
\begin{array}{c| c c c c}
x & \delta^{+}(x) & \delta^{-}(x) & \delta^{+}_2(x) & \delta^{-}_2(x) \\
\hline
10^{7} & 0.7418 & 0.6739 & 0.6673 & 0.7103 \\
10^{8} & 0.7203 & 0.6664 & 0.6401 & 0.6803 \\
10^{9} & 0.7026 & 0.6585 & 0.6186 & 0.6560 \\
10^{10} & 0.6875 & 0.6506 & 0.6001 & 0.6349 \\
\end{array}
\]
\caption{(Type-III bias) Variation of functions $\delta^{\pm}(x)$ and $\delta^{\pm}_2(x)$ are very slow and smooth. 
Values of the functions are rounded off to the four decimal places.}
\label{tab3}
\end{table}
Similarly, the function $\delta^{-}_2(x)$ is defined to quantify the bias in the twin prime gap minus 1. 
If the prime pairs were distributed in a completely random manner, both $\delta^{\pm}_2(x)$ would have been less than 0.5 since we have more 
odd composites than the primes. But what we find is that $\delta^{-}_2(x) > \delta^{+}_2(x) > 0.5$ for all the values of $x$ that we investigated.
Exactly similar bias is seen for the prime numbers, but in this case, interestingly, $\delta^{+}(x) > \delta^{-}(x) > 0.5$. 
The values of the functions $\delta^{\pm}_2(x)$ for some values of $x$ can be found in Table \ref{tab3}.
The results related to this particular type of bias can be seen in Fig. \ref{fig3}. 

\begin{figure}[tb]
\includegraphics[width=0.80\linewidth]{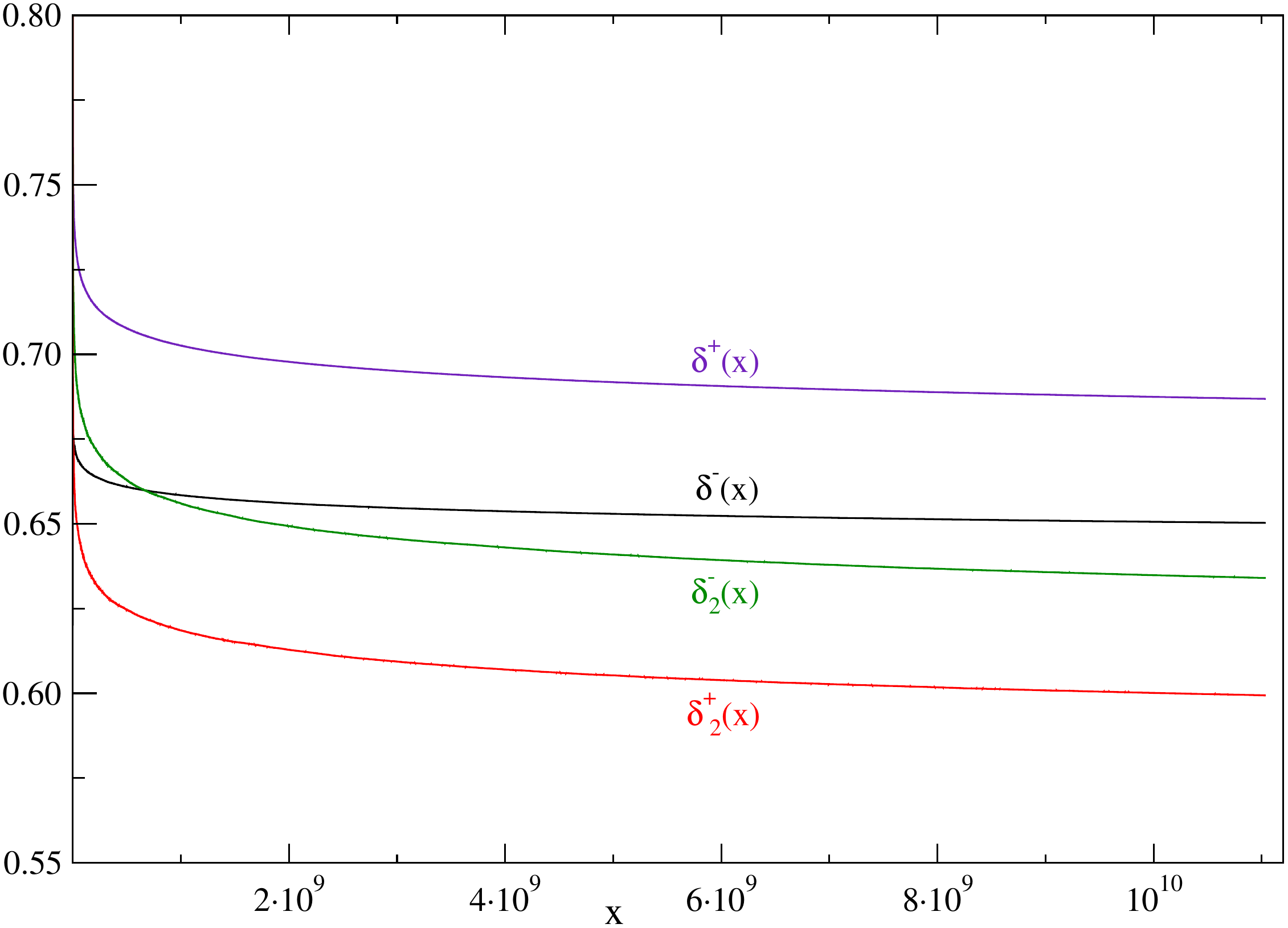}
\caption{(Type-III bias) Plots of the functions $\delta^{+}_2(x)$ and $\delta^{-}_2(x)$. The corresponding results for the prime numbers are also 
provided.}
\label{fig3}
\end{figure}

Although $\delta^{\pm}(x)$ and $\delta^{\pm}_2(x)$ change very slowly, we expect all these functions to eventually go to 0 since the number 
of odd composite numbers grow faster than the prime numbers. Assuming infinitude of twin primes, we propose the following conjecture for this 
type of bias.
\begin{conjecture}
As $x \rightarrow \infty$, $\delta^{+}(x) = \delta^{-}(x) = 0$ and $\delta^{+}_2(x) = \delta^{-}_2(x) = 0$.
\label{cj10.2}
\end{conjecture}


\bibliographystyle{amsplain}

\end{document}